\documentclass[sn-mathphys-num]{sn-jnl}
\usepackage{graphicx}%
\usepackage{subfigure}
\usepackage{multirow}%
\usepackage{amsmath,amssymb,amsfonts,}%
\usepackage{amsthm}%
\usepackage{mathrsfs}%
\usepackage[title]{appendix}%
\usepackage{xcolor}%
\usepackage{textcomp}%
\usepackage{manyfoot}%
\usepackage{booktabs}%
\usepackage{algorithm}%
\usepackage{algorithmicx}%
\usepackage{algpseudocode}%
\usepackage{listings}%
\theoremstyle{thmstyleone}%
\newtheorem{theorem}{Theorem}
\newtheorem{proposition}[theorem]{Proposition}% 
\newtheorem{lemma}[theorem]{Lemma}
\newtheorem{corollary}[theorem]{Corollary}
\theoremstyle{thmstyletwo}%
\newtheorem{remark}{Remark}%

\theoremstyle{thmstylethree}%
\newtheorem{definition}{Definition}%

\newcommand{\R}{\mathbb{R}}

\renewcommand{\phi}{\varphi}

\newcommand{\Ric}{\operatorname{Ric}}
\newcommand{\tr}{\operatorname{tr}}

\renewcommand{\d}{\partial}

\begin{document}

\title[]{Almost coK\"{a}hler manifolds in the context of mixed Killing vector field}
\author{\fnm{Paritosh} \sur{Ghosh}} \email{paritoshg.math.rs@jadavpuruniversity.in}
%\author{\fnm{Arindam} \sur{Bhattacharyya}} \email{arindambhattacharyya.jadavpuruniversity.in}
\affil{\orgdiv{Department of Mathematics}, \orgname{Jadavpur University}, \orgaddress{\street{Jadavpur}, \city{Kolkata}, \postcode{700032}, \state{West Bengal}, \country{India}}}

\abstract{A vector field \unboldmath{$V$} on any (semi-)Riemannian manifold is said to be {\it mixed Killing} if for some nonzero smooth function $f$, it satisfies $L_VL_Vg=fL_Vg$, where $L_V$ is the Lie derivative along $V$. This class of vector fields, as a generalization of Killing vector fields, not only identify the isometries of the manifolds, but broadly also contain the class of homothety transformations. We prove an essential curvature identity along those fields on any (semi-)Riemannian manifold and thus generalize the Bochner's theorem for Killing vector fields in this setting. Later we study it in the framework of almost coK\"{a}hler structure and we prove that the Reeb vector field \unboldmath{$\xi$} on an almost coK\"{a}hler manifold is mixed Killing if and only if the operator $h=0$. Moving further, we completely classify almost coK\"{a}hler manifolds with $\xi$ mixed Killing vector field in dimension 3. In particular, if $\xi$ on an $\eta$-Einstein almost coK\"{a}hler manifold is mixed Killing, then the manifold is of constant scalar curvature with $h=0$. Also we show that on any $(\kappa,\mu)$-almost coK\"{a}hler manifold, $\xi$ is mixed Killing if and only if the manifold is coK\"{a}hler. In the end we present few model examples in this context.}
%The concept of mixed Killing vector field is introduced and studied as a generalization of Killing vector fields in this paper.} 

\pacs[MSC Classification 2020]{53C25, 53D15}
\keywords{Mixed Killing vector field, almost coK\"{a}hler manifold, coK\"{a}hler manifold, \unboldmath{$(\kappa,\mu)$}-almost coK\"{a}hler manifold}

\maketitle
\section{Introduction}
 Metric tensor preserving vector fields on a (semi-)Riemannian manifold $(M,g)$ were introduced as Killing vector fields. Equivalently, a vector field $V$ on $(M,g)$ is Killing if $L_Vg=0$, where $L_V$ denotes the Lie derivative along $V$. Killing vector fields predict the symmetries of the manifolds and they are naturally Jacobi vector fields along every geodesic $\gamma(t), ~~t\in\R$ (satisfies $\nabla_{\gamma'}\nabla_{\gamma'}V+R(V,\gamma')\gamma'=0$, $R$ is Riemann curvature tensor) \cite{BN}. A celebrated theorem of Bochner \cite{Bo} regarding Killing vector field says that 
\vspace{0.2cm}
\begin{theorem}
    Let $V$ be a Killing vector field on a compact Riemannian manifold $M$ and along $V$ Ricci curvature is $\Ric(V,V)\le0$, then $V$ is parallel. In fact, if Ricci curvature is negative definite, then Killing vector fields are nothing but zero vector field.
\end{theorem}
\vspace{0.2cm}
\noindent
Author in \cite{O} introduced and studied 2-Killing vector fields as a vector field $V$ satisfying $L_VL_Vg=0$. The author also presented a necessary and sufficient curvature condition for $V$ to be 2-Killing (see {\it Theorem 2.1 of \cite{O}}). Various geometers since then studied these types of vector fields in almost contact structure \cite{BO2} and in warped product spaces \cite{BO, SU}. 
\vspace{0.2cm}

In this article, we classify a special class of vector fields, namely {\it mixed Killing vector fields} on any (semi-)Riemannian manifold and study the curvature properties for such manifolds.
\vspace{0.2cm}

\begin{definition}
    A vector field $V$ on a (semi-)Riemannian manifold $(M,g)$ is said to be {\it mixed Killing vector field} if there exists non-zero smooth function $f$ on $M$ such that $V$ satisfies 
\begin{equation}\label{e1}
    L_VL_Vg=fL_Vg.
\end{equation}
We call $f$ as {\it mixed Killing factor} of $V$. 
\end{definition}
\vspace{0.2cm}
\noindent
Throughout this paper we denote $f$ as a mixed Killing factor. Note that such $f$ is unique up to non-Killing (and hence non-2-Killing) vector field. If $V$ is mixed Killing vector field, then $V$ is Killing if and only if $V$ is 2-Killing. In a natural way, mixed Killing vector fields are generalizations of Killing vector fields. But a mixed Killing vector field may not be Killing, for instance hemothetic vector field. A vector field $V$ is homothetic if $L_Vg=2\lambda g$, for some constant $\lambda$. For $\lambda\ne 0$, homothetic vector fields are mixed Killing for non-zero constant function $f=2\lambda$. Thus it generalizes the class of conformal vector field on the manifold. It is noteworthy that, a conformal vector field $V$ with conformal factor $\lambda$ (that is, $L_Vg=2\lambda g$) is mixed Killing with $f=2\lambda$ if and only if $\lambda$ is constant along the flow of $V$.

We call a mixed Killing vector field $V$ {\it proper} if $V$ is not Killing. Proper mixed Killing vector fields $V$ generate affine transformations, under which the metric doesn't remain constant, rather it is proportional to the rate of change of the metric along $V$. The flow of such $V$, which properly contains the homotheties, places a high constraints on the curvature properties of the manifold.

This motivates us to study the mixed Killing vector fields extensively. This article is organized in the following manner. In Section \ref{sec2}, we establish the curvature properties of the Riemannian manifolds having mixed Killing vector fields. Proceeding further we study these vector field in the framework of a special types of almost contact metric structure, namely almost coK\"{a}hler structure. We present a brief overview of almost coK\"{a}hler manifolds in Section \ref{sec3}. Section \ref{sec4} is devoted in the classifications of almost coK\"{a}hler manifolds with the Reeb vector field $\xi$ as mixed Killing vector field. In fact we show that on any almost coK\"{a}hkler manifold, $\xi$ is mixed Killing if and only if $\xi$ is parallel (see Theorem \ref{t1}). Therefore we completely characterize three dimensional almost coK\"{a}haler manifolds with $\xi$ mixed Killing (see Theorem \ref{t7}). We also find the necessary condition for any vector fields, which are pointwise collinear with $\xi$ to be mixed Killing vector field. Moving ahead we characterize $\eta$-Einstein almost coK\"{a}hler manifolds and $(\kappa,\mu)$-almost coK\"{a}hler manifolds when $\xi$ is mixed Killing in subsection \ref{ssec1} and \ref{ssec2} respectively. We also deploy few model examples in subsection \ref{ssec3}. Finally we conclude our article with few summary statements and future scope of this study.

\vspace{0.3cm}

\section{On mixed Killing vector field}\label{sec2}
Covariant second derivative is explicitly connected to the curvature tensor of any (semi-)Riemannian manifold. Curvature tensor plays the pivotal role in studying the geometry. Therefore finding such mixed Killing vector fields pass the information for the geometry of the manifold. For instance, existence of non-trivial conformal vector field on $(M,g)$ may assert the fact that $M$ is isometric to either sphere or flat space. Where as complex hyperbolic space doesn't contain any non-trivial conformal vector field \cite{SS}.

\vspace{0.3cm}
\begin{theorem}
    A mixed Killing vector field $V$ on $(M,g)$ is invariant under a conformal change of the metric $g\to g_\rho=\rho g$ if and only if $2(V\rho)L_Vg=[f(V\rho)-V(V\rho)]g,$ where $\rho$ is a positive smooth function and $f$ is the mixed Killing factor.
\end{theorem}
\begin{proof}
    For $g_\rho=\rho g$, we have by a straight forward computation
    \begin{eqnarray*}
        L_Vg_\rho&=&\rho L_Vg+(V\rho)g\\
        L_VL_Vg_\rho&=&\rho L_VL_Vg+2(V\rho)L_Vg+V(V\rho)g.
    \end{eqnarray*}
    Therefore $V$ is mixed Killing on $(M,g_\rho)$ if and only if $L_VL_Vg_\rho=fL_Vg_\rho$, that is, if and only if $2(V\rho)L_Vg=[f(V\rho)-V(V\rho)]g$. 
\end{proof}

\vspace{0.3cm}

Let $\omega$ be the dual 1-form of the smooth vector field $V$, that is $\omega(X)=g(X,V)$ on $(M,g)$ for all $X\in\chi(M)$, collection of all vector fields on $M$. Then it is easy to see that 
\begin{eqnarray*}
    2d\omega(X,Y)&=&X\omega(Y)-Y\omega(X)-\omega([X,Y])\\
    &=&g(\nabla_XV,Y)-g(X,\nabla_YV).
\end{eqnarray*}
Now we define a skew-symmetric $(1,1)$ tensor field $\phi$ on $(M,g)$ by 
\begin{eqnarray}\label{e2}
    g(X,\phi Y)=2d\omega(X,Y)=g(\nabla_XV,Y)-g(X,\nabla_YV).
\end{eqnarray}
Using this we can calculate 
\begin{eqnarray}\label{e3}
    (L_Vg)(X,Y)&=&g(\nabla_XV,Y)+g(X,\nabla_YV)\nonumber\\
    &=&g(2\nabla_YV+\phi Y, X).
\end{eqnarray}
Recall that 
\begin{eqnarray}\label{e4}
    L_V((L_Vg)(X,Y))=(L_VL_Vg)(X,Y)+(L_Vg)(L_VX,Y)+(L_Vg)(X,L_VY).
\end{eqnarray}
Plugging \eqref{e3} in the above one can then have
\begin{eqnarray}\label{e5}
    (L_VL_Vg)(X,Y)&=&2g(\nabla_V\nabla_YV,X)-2g(\nabla_{\nabla_VY}V,X)+3g(\phi\nabla_YV,X)\nonumber\\
    &+&g((\nabla_V\phi)Y,X)+4g(\nabla_{\nabla_YV}V,X)+g(\phi^2Y,X)+g(\nabla_{\phi Y}V,X).
\end{eqnarray}
Note that, the Riemann curvature tensor $R$ can be written as 
    \begin{equation*}
        R(V,Y)=[\nabla_V,\nabla_Y]-\nabla_{[V,Y]}.
    \end{equation*}
 Availing this in \eqref{e4}, we acquire
 \begin{eqnarray}\label{e6}
     (L_VL_Vg)(X,Y)&=&2g(R(V,Y)V,X)+2g(\nabla_Y\nabla_VV,X)+2g(\nabla_{\nabla_YV}V,X)\nonumber\\
     &+&3g(\phi\nabla_YV,X)+g(\phi^2Y,X)+g(\nabla_{\phi Y}V,X)
 \end{eqnarray}
Therefore the following proposition can be stated.
\vspace{0.2cm}
\begin{proposition}
    A vector field $V$ is mixed Killing if and only if it satisfies, for any vector field $Y$ on $(M,g)$,
    \begin{eqnarray}\label{e7}
        2R(V,Y)V&=&2f\nabla_YV+f\phi Y-2\nabla_Y\nabla_VV-2\nabla_{\nabla_YV}V\nonumber\\
        &-&3\phi\nabla_YV-\phi^2Y-\nabla_{\phi Y}V.
    \end{eqnarray}
\end{proposition}
\vspace{0.2cm}

\begin{remark}
    As $\phi$ is skew-symmetric, $g(\phi Y,Y)=0$. Also using \eqref{e2}, one can see
    \begin{eqnarray*}
        g(\nabla_{\nabla_YV}V,Y)&=&g(\nabla_YV,\phi Y)+g(\nabla_YV,\nabla_YV),\\
        g(\nabla_{\phi Y}V,Y)&=&g(\nabla_YV,\phi Y)-g(Y,\phi^2Y). 
    \end{eqnarray*}
    But then inner product with $Y$ of \eqref{e6} and making use of the above, we get for all $Y$ in $\chi(M)$, 
    \begin{equation}\label{e8}
        g(R(Y,V)V,Y)=g(\nabla_Y\nabla_VV,Y)+g(\nabla_YV,\nabla_YV)-fg(\nabla_YV,Y).
    \end{equation}
    This equation can also be treated as a necessary and sufficient condition for a vector field $V$ to be mixed Killing.
\end{remark}
\vspace{0.2cm}
\noindent
In contrast to Bochner's theorem \ref{t1} for Killing vector fields, we present a characterization of non-zero mixed Killing vector fields on compact Riemannian manifold with negative definite Ricci curvature. In this context we prove the following.
\vspace{0.2cm}
\begin{theorem}
    Let $V$ be a mixed Killing vector field on a compact Riemannian manifold $(M^n,g)$ and along $V$ Ricci curvature is $\Ric(V,V)\le0$. Then $V$ is parallel (hence Killing) whenever $\int_M|\nabla V|^2dvol_g\ge\int_Mf\operatorname{div}Vdvol_g$. In particular, $V$ is parallel whenever either $f$ is constant or $\operatorname{div}V=0$. 
\end{theorem}
\begin{proof}
    Consider an orthonormal frame of vector fields $\{e_1,e_2,...,e_n\}$ of $T_pM$ at $p\in M$. But then plugging $Y=e_i$ in \eqref{e7} and summing over $i$, we have
    \begin{equation}\label{e9}
        \Ric(V,V)=\operatorname{div}{\nabla_VV}+|\nabla V|^2-f\operatorname{div}{V}.
    \end{equation}
    We now integrate the above over the compact manifold $M$ and applying Green's theorem we obtain
    \begin{equation*}
        \int_M\Ric(V,V)dvol_g=\int_M|\nabla V|^2dvol_g-\int_Mf\operatorname{div}Vdvol_g
    \end{equation*}
    Then $\Ric(V,V)\le 0$ implies $\nabla V=0$ whenever $\int_M|\nabla V|^2dvol_g\ge\int_Mf\operatorname{div}Vdvol_g$.
\end{proof}

\begin{remark}\label{rem2}
    Observe that, if $V=r(x)\frac{d}{dx}$ be an vector field on the real line $\mathbb{R}$ with metric $g=dx^2$, then 
    \begin{eqnarray*}
        (L_Vg)_{00}=2r'(x)~~~\text{and}~~~ (L_VL_Vg)_{00}=2r(x)r''(x)+4r'(x)^2,
    \end{eqnarray*}
    where $r'(x)=\frac{dr}{dx}$ and $r''(x)=\frac{d^2r}{dx^2}$. Note also that, if $x(t)$ is the integral curve of $V$, then $\frac{dx}{dt}=r(x)$. Hence one can write 
    \begin{eqnarray*}
        r'(x)&=&\frac{1}{r(x)}\frac{dr}{dt},\\
        r''(x)&=&-\frac{1}{r(x)^3}\bigg(\frac{dr}{dt}\bigg)^2+\frac{1}{r(x)^2}\frac{d^2r}{dt^2}.
    \end{eqnarray*}
    Moreover, if $V$ is the mixed Killing vector field on the real line with mixed Killing factor $f$ then we have
    $$r(x)r''(x)+2r'(x)^2=2fr'(x).$$
    Availing the values of $r'(x)$ and $r''(x)$, we obtain 
    $$r\frac{d^2r}{dt^2}+\frac{dr}{dt}=2fr\frac{dr}{dt}.$$
    But then we have the solution as
    $$r^2=c\int\exp\big({2\int f dt}\big)dt+c',$$
    for some constants $c,c'$. In particular, if $f$ is constant, then
    $$r^2=\frac{c}{f}\exp(2ft)+c'.$$ 
    This clearly predicts the nonlinearity of the mixed Killing vector fields in higher dimensions (even in Euclidean spaces).
\end{remark}
\vspace{0.3cm}

\section{Almost coK\"{a}hler manifolds: An overview}\label{sec3}
We give a brief overview of almost coK\"{a}hler manifold and recall some results in this section.  An {\it almost contact metric manifold} $M^{2n+1}(\phi,\xi,\eta,g)$ is an odd dimensional (say ${2n+1}$, $n\ge1$) Riemannian manifold $(M,g)$ associated with a {\it characteristic vector field} $\xi$, a global 1-form $\eta$, a structural $(1,1)$ type tensor field $\phi$ and a {\it compatible} metric $g$ such that the following equations hold true for any $Y,X\in \chi(M)$:
\begin{eqnarray}
    \phi^2Y=-Y+\eta(Y)\xi,~~~\eta(\xi)=1,\label{e10}\\
    g(\phi Y,\phi X)=g(Y,X)-\eta(Y)\eta(X).\label{e11}
\end{eqnarray}
 Vector field $\xi$ is called the {\it Reeb vector field}. An almost contact metric manifold is said to be {\it normal} if the {\it Nijenhuis torsion} of type $(1,2)$ of the tensor field $\phi$, given by $[\phi,\phi](X,Y)=\phi^2[X,Y]+[\phi X,\phi Y]-\phi[\phi X,Y]-\phi[X,\phi Y]$ is equal to $-2d\eta(X,Y)\xi$. Define {\it fundamental 2-form}, $\Phi$ on almost contact metric manifold by $\Phi(X,Y)=g(X,\phi Y)$. \par
If both $\eta$ and $\Phi$ are closed, that is $d\eta=0$ and $d\Phi=0$, then the almost contact metric manifold is called {\it almost coK\"{a}hler manifold}. Various authors \cite{DO,Ols1,Ols2,Ols3} studied (almost) coK\"{a}hler manifold in the name of (almost) cosymplectic manifold, and let us take this opportunity to point out that these are the same. A normal almost coK\"{a}hler manifold is called {\it coK\"{a}hler manifold}. It is shown that an almost contact metric manifold is coK\"{a}hler if and only if $\nabla\phi=0$ \cite{Bla}. In other way to say that, when both $\eta$ and $\phi$ are parallel, almost coK\"{a}hler manifold is coK\"{a}hler. 

To proceed further, define two traceless symmetric structural $(1,1)$ tensor field $h$ and $h'$ by $h=\frac{1}{2}L_\xi\phi$ and $h'=h\phi$. Then the following relations hold true on an almost coK\"{a}hler manifold \cite{Ols2,Ols3}.
\begin{align}
    &h\xi=0,~~~h\phi=-\phi h,~~~\nabla_\xi\phi=0\label{e12}\\
    &\nabla\xi=h',~~~(\nabla_Y\eta)X=g(h'X,Y),~~~\operatorname{div}\xi=0, \label{e13}\\
    &R(Y,X)\xi=(\nabla_Yh')X-(\nabla_Xh')Y,\label{e14}\\
    &\Ric(\xi,\xi)=-\tr(h^2).\label{e15}
\end{align}
\noindent
Note that, $2n$ dimensional distribution $\mathcal{D}=\ker\eta=\{Y\in\chi(M):g(Y,\xi)=0\}$ of an almost coK\"{a}hler manifold $M^{2n+1}$ is integrable and hence defines a foliation perpendicular to the Reeb foliation. According to Olszak \cite{Ols1}, {\it an almost coK\"{a}hler manifold is said to have K\"{a}hlerian leaves if the leaves of the distribution $\mathcal{D}$ are K\"{a}hler manifolds}. 
He proved that {\it an almost coK\"{a}hler manifold $M$ is an almost coK\"{a}hler with K\"{a}hlerian leaves if and only if $(\nabla_Y\phi)X=g(X,hY)\xi-\eta(X)hY$ holds on $M$}. In fact he proved that an almost coK\"{a}hler manifolds with K\"{a}hlerian leaves is coK\"{a}hler if and only if $\xi$ is parallel (i.e. $h=0$). As a 3 dimensional almost coK\"{a}hler manifold has K\"{a}hlerian leaves, hence it is coK\"{a}hler if and only if $h=0$ (hence $\xi$ is parallel) \cite{Per}.\par
We say that an almost coK\"{a}hler manifold $M^{2n+1}(\phi,\xi,\eta,g)$ is  {\it$\eta$-Einstein} if the Ricci operator $Q$ of $M$ satisfies 
\begin{equation}\label{e16}
    QY=aY+b\eta(Y)\xi,
\end{equation}
where the smooth functions $a$ and $b$ are given by $a=\frac{1}{2n}\big(r+\tr(h^2)\big)$ and $b=-\frac{1}{2n}\big(r+(2n+1)\tr(h^2)\big)$.

An almost coK\"{a}hler manifold $M^{2n+1}(\phi,\xi,\eta,g)$ is said to be {\it $(\kappa,\mu)$-almost coK\"{a}hler manifold} for some smooth functions $\kappa,\mu$, if the Riemann curvature tensor of $M$ satisfies,  
\begin{equation}\label{e17}
    R(X,Y)\xi=\kappa(\eta(Y)X-\eta(X)Y)+\mu(\eta(Y)hX-\eta(X)hY).
\end{equation}
Particularly, if $\mu=0$, $\xi$ belongs to {\it $\kappa$-nullity distribution}, a subbundle $N(\kappa)$ defined by for $p\in M$ \cite{Bla} $$N_p(\kappa)=\{Z\in T_pM:R(X,Y)Z=\kappa(g(Y,Z)X-g(X,Z)Y), X,Y\in T_pM\}.$$
And it can be shown that $\kappa=0$ if and only if $\xi$ is Killing.
One can see on $(\kappa,\mu)$-almost coK\"{a}hler manifold that
\begin{equation}\label{e18}
    Q\xi=2n\kappa\xi~~\text{and}~~h^2=\kappa\phi^2.
\end{equation}
Hence $\kappa\le 0$ and $M$ is coK\"{a}hler if and only if $\kappa=0$ \cite{CDY}. We say a $(\kappa,\mu)$-almost coK\"{a}hler manifold is {\it non-coK\"{a}hler} for $\kappa<0$. 

$\mathcal{D}$-homothetic deformation was introduced \cite{T} in search of a different contact metric structure deforming an existing structure and under that structure the contact subbundle $\mathcal{D}$ is homothetic. Consider an almost coK\"{a}hler manifold $M(\phi,\xi,\eta,g)$ and a smooth function $u\ne0$ varying along $\xi$ only. Then one can always find a different structure $(\phi',\xi',\eta',g')$ on $M$ defined by \cite{DO}
\begin{equation*}
    \phi'=\phi,~~~\xi'=\frac{1}{u}\xi,~~~\eta'=u\eta,~~~g'=cg+(u^2-c)\eta\otimes\eta,
\end{equation*}
$c$ is some positive constant.
This transformation is known as $\mathcal{D}$-homothetic deformation.

Here it is worth mentioning that {\it a $(\kappa,\mu,\nu)$-almost coK\"{a}hler manifold with $\kappa<0$ can be transformed $\mathcal{D}$-homothetically to a $(-1,\frac{\mu}{\sqrt{-\kappa}},0)$-almost coK\"{a}hler manifold} (see Proposition 7 of \cite{DO}). Also note that \cite{DO}, if $\kappa<0$, then $\kappa$ is constant if and only if $\nu=0$. Therefore for a non-coK\"{a}hler $(\kappa,\mu)$-almost coK\"{a}hler manifold, $\kappa$ is constant. We refer \cite{CDY,DO,GB} for more studies on $(\kappa,\mu)$-almost coK\"{a}hler manifold.

\vspace{0.2cm}
\noindent
Note that the Ricci operator $Q$ of a 3 dimensional  $(\kappa,\mu)$-almost coK\"{a}hler manifold can be written as \cite{CM}
\begin{equation}\label{e20}
    QY=\mu hY+\bigg(\frac{r}{2}-\kappa\bigg)Y+\bigg(3\kappa-\frac{r}{2}\bigg)\eta(Y)\xi.
\end{equation}
In {\it Theorem 1 of \cite{DO}}, authors proved that non-coK\"{a}hler $(\kappa,\mu)$-almost coK\"{a}hler manifolds are almost coK\"{a}hler with K\"{a}hlerian leaves. 
%We also recall following two identities from \cite{OAM}:
%\begin{eqnarray}
 %   (\nabla_Y\phi)X&=&g(X,hY)\xi-\eta(X)hY\label{e21}\\
  %  (\nabla_Yh)X-(\nabla_Xh)Y&=&\kappa(\eta(X)\phi Y-\eta(Y)\phi X+2g(X,\phi  Y)\xi)\nonumber\\
  %  &+&\mu(\eta(X)\phi hY -\eta(Y)\phi hX).\label{e22}
%\end{eqnarray}
%In view of the first equation, it can be said that a
Therefore a $(\kappa,\mu)$-almost coK\"{a}hler manifold is coK\"{a}hler if and only if $\xi$ is parallel $(h=0)$ \cite{Ols1}. %Also we can confirm $\tr\phi$ to be constant from \eqref{e21}.

\vspace{0.3cm}

\section{Almost coK\"{a}hler manifolds with $\xi$ mixed Killing}\label{sec4}
It is known that on any almost coK\"{a}hler manifold $M(\phi,\xi,\eta,g)$, the Reeb vector field $\xi$ is Killing if and only if $h=0$. In this section, we generalize this in terms of mixed Killing vector field and we show that 

\vspace{0.2cm}

\begin{theorem}\label{t1}
    On any almost coK\"{a}hler manifold $M(\phi,\xi,\eta,g)$, $\xi$ is mixed Killing  if and only if $h=0$. Moreover, $M$ is locally a Riemannian product of Real line and an almost K\"{a}hler manifold. 
\end{theorem}

\vspace{0.2cm}
\noindent
But before going into the proof of this theorem, we state the following lemma.
\vspace{0.2cm}
\begin{lemma}\label{l1}
    On any almost coK\"{a}hler manifold $M(\phi,\xi,\eta,g)$, the followings hold.
    \begin{enumerate}
        \item $(L_\xi g)(X,Y)=2g(h'X.Y)$,
        \item $(L_\xi L_\xi g)(X,Y)=4g(h^2X,Y)+2g((\nabla_\xi h')X,Y)$.
    \end{enumerate}
\end{lemma}
\begin{proof}
    The proof of first one is straight forward using first result of \eqref{e13}. While for the second, we use the first part and \eqref{e4} for $V=\xi$ to compute
    \begin{eqnarray*}
        (L_\xi L_\xi g)(X,Y)=4g(h'^2X,Y)+2g((\nabla_\xi h')X,Y).
    \end{eqnarray*}
    Observe that $h'^2=h^2$ and hence we have the required result.
\end{proof}

\begin{proof}[Proof of Theorem \ref{t1}]
Consider $\xi$ to be mixed Killing vector field on $M$. Then for some mixed Killing factor $f$, we have
$$L_\xi L_\xi g=fL_\xi g.$$
Then using the above lemma, we obtain 
\begin{equation}\label{e23}
    \nabla_\xi h'=fh'-2h^2.
\end{equation}
It is known from \eqref{e14} that $$R(Y,\xi)\xi=(\nabla_Yh')\xi-(\nabla_\xi h')Y.$$
As $h'\xi=0$ and $h'^2=h^2$ we have $(\nabla_Yh')\xi=-h^2Y$. Availing this and \eqref{e23} in the above, we then obtain
    \begin{equation}\label{e24}
        R(Y,\xi)\xi=-fh\phi Y+h^2Y.
    \end{equation}
But again contraction of this with respect to $Y$ gives $\Ric(\xi,\xi)=\tr h^2$. Comparing this with \eqref{e15}, we can affirm $\tr h^2=0$. Since $h$ is symmetric, we confirm $h=0.$ Hence $\xi$ is parallel and therefore Killing. 
 Moreover, in view of {\it Theorem 3.11 of \cite{CDY}}, $M$ is locally a Riemmanian product of real line and an almost K\"{a}hler manifold. 
\end{proof}

\vspace{0.2cm}

As a consequence of the above theorem, as $h=0$ we have $R(Y,X)\xi=0$ for all $Y,X\in\chi(M)$. Taking covariant derivative of this along $Z\in\chi(M)$, $(\nabla_ZR)(Y,X)\xi=0.$ In other words, the scalar curvature $r$ of the manifold $M$ is constant along $\xi$. Particularly in dimension 3, Weyl conformal curvature tensor $W$ vanishes identically, where 
\begin{eqnarray*}
    W(X,Y)Z&=&R(X,Y)Z-[\Ric(Y,Z)X-\Ric(X,Z)Y+g(Y,Z)QX-g(X,Z)QY]\nonumber\\
    &+&\frac{r}{2}[g(Y,Z)X-g(X,Z)Y].
\end{eqnarray*}
Therefore, Riemann curvature tensor can be given by
\begin{eqnarray*}
    R(X,Y)Z&=&[\Ric(Y,Z)X-\Ric(X,Z)Y+g(Y,Z)QX-g(X,Z)QY]\\
    &-&\frac{r}{2}[g(Y,Z)X-g(X,Z)Y].
\end{eqnarray*}
But then for $Z=\xi$ and because $\xi$ is mixed Killing, we have
$$\eta(Y)QX-\eta(X)QY=\frac{r}{2}[\eta(Y)X-\eta(X)Y].$$
We now replace $X$ by $\phi X$ to get $Q\phi X=\frac{r}{2}\phi X.$ Since $Q\xi=0$, we finally have $QX=\frac{r}{2}(X-\eta(X)\xi)$. Hence the manifold is $\eta$-Einstein. Moreover, such manifold has constant scalar curvature $r$, as proved in Theorem \ref{t10}. Also the manifold is locally symmetric (satisfies $(\nabla_YQ)X=(\nabla_XQ)Y$). Then according to {\it Proposition 3.1 of \cite{Per2}},\\ $M^3$ is locally isometric to either of these 
    \begin{enumerate}
        \item Euclidean space $\mathbb{R}^3$
        \item  Product of 1 dimensional manifold and a 2 dimensional K\"{a}hler surface with constant curvature. 
    \end{enumerate}

\vspace{0.2cm}

\begin{theorem}\label{t7}
    Consider an almost coK\"{a}hler manifold $M^3(\phi,\xi,\eta,g)$ with $\xi$ mixed Killing. Then $M^3$ is coK\"{a}hler and it is isometric to either $\mathbb{R}^3$ or $\mathbb{R}\times K^2$, $K^2$ being the 2 dimensional K\"{a}hler surface of constant curvature.
\end{theorem}

\vspace{0.2cm}
\noindent
Recall from previous section that an almost coK\"{a}hler manifold with K\"{a}hlerian leaves is coK\"{a}hler if and only if $h=0$. A direct consequence of this statement can be asserted as follows.
\vspace{0.2cm}

\begin{theorem}\label{t8}
    Consider an almost coK\"{a}hler manifold $M(\phi,\xi,\eta,g)$ with K\"{a}hlerian leaves. Then $M$ is coK\"{a}hler if and only if $\xi$ is mixed Killing.
\end{theorem}
\vspace{0.2cm}

\begin{remark}\label{rem3}
    We also mention that $\xi$ on an almost coK\"{a}hler manifold is 2-Killing if and only if $\nabla_\xi h'=2h^2$. In view of the second and third equation of \eqref{e12}, one can see 
    \begin{eqnarray*}
        \nabla_\xi h'=\nabla_\xi(h\phi)=-\nabla_\xi(\phi h)=-\phi(\nabla_\xi h). 
    \end{eqnarray*}
    Therefore, $\xi$ is 2-Killing if and only if $-\phi(\nabla_\xi h)=-2h^2$, that is $\nabla_\xi h=-2\phi h^2.$
\end{remark}

\vspace{0.2cm}

\begin{remark}
    Let $(\phi',\xi',\eta',g')$ be a structure of $M(\phi,\xi,\eta, g)$ under the $\mathcal{D}$-homothetic deformation, mentioned in the previous section.
If $H=\frac{1}{2}L_{\xi'}\phi'$ and $R'$ being the Riemann curvature tensor defined on $M(\phi',\xi',\eta',g')$, then \cite{DO} 
\begin{eqnarray*}
    &H=\frac{1}{u}h,~~~H'=H\phi'=\frac{1}{u}h'\nonumber\\
    &R'(Y,X)\xi'=\frac{1}{u}R(Y,X)\xi+\frac{\xi u}{u^2}(\eta(X)h'Y-\eta(Y)h'X).
\end{eqnarray*}
Hence, $h=0$ if and only if $H=0$. Therefore the property of the Reeb vector field to be mixed Killing is invariant under the $\mathcal{D}$-homothetic deformation. In other words, $\xi$ is mixed Killing if and only if $\xi'$ is mixed Killing.
\end{remark}

\vspace{0.2cm}

In the next theorem, we find out the necessary condition for a vector field $V$ to be mixed Killing where $V$ is pointwise collinear with $\xi$.  
\vspace{0.2cm}

\begin{theorem}
    Consider an almost coK\"{a}hler manifold $M(\phi,\xi,\eta,g)$ and the vector field $V=\alpha\xi$, for some smooth $\alpha$ on $M$. If $V$ is mixed Killing, then 
    \begin{enumerate}
        \item $\alpha^2(\nabla_\xi h)X=(\phi X\alpha)(\operatorname{grad}\alpha-(\xi\alpha)\xi)+\alpha^2\phi h^2X+\alpha(f+\alpha-\xi\alpha)hX.$ \label{o1}
        \item $(\xi\alpha)^2=|\operatorname{grad}\alpha|^2+\alpha^2\tr h^2.$\label{o2}
    \end{enumerate}
\end{theorem}
\begin{proof}
    Let $V$ be a mixed Killing vector field on $M$ with mixed Killing factor $f$. Then we have for all vector fields $X,Y$ on $M$,
    $$(L_VL_Vg)(X,Y)=f(L_Vg)(X,Y).$$
    Since $V=\alpha\xi$ for some non-zero smooth function $\alpha$, we have $\nabla_XV=(X\alpha)\xi+\alpha h'X.$ Using first equation of \eqref{e13}, we have
    \begin{eqnarray}\label{e25}
        (L_Vg)(X,Y)=(X\alpha)\eta(Y)+(Y\alpha)\eta(X)+2\alpha g(h'X,Y).
    \end{eqnarray}
    Plugging this in \eqref{e4}, we obtain
    \begin{eqnarray}\label{e26}
        (L_VL_Vg)(X,Y)&=&X(V\alpha)\eta(Y)+Y(V\alpha)\eta(X)+(X\alpha)(L_V\eta)Y+(Y\alpha)(L_V\eta)X\nonumber\\
        &+&2(V\alpha)g(h'X,Y)+2\alpha g((L_Vh')X,Y)+4\alpha^2g(h^2X,Y)\nonumber\\
        &+&3\alpha (h'X\alpha)\eta(Y)+\alpha(h'Y\alpha)\eta(X).
    \end{eqnarray}
    One can easily verify the following computations.
    \begin{eqnarray}
        X(V\alpha)&=&(X\alpha)(\xi\alpha)+\alpha X(\xi\alpha).\label{e27}\\
        (L_V\eta)X&=&L_V(\eta(X))-\eta(L_VX)\nonumber\\
        &=&(\nabla_V\eta)X+\eta(\nabla_XV)\nonumber\\
        &=&X\alpha. ~~~~~~~~~~~~~~~~~~~~~~~~~~~~~~~~~~~~~~~~~~\text{[since $h'\xi=0$]}\label{e28}\\
        (L_Vh')X&=&L_V(h'X)-h'L_VX\nonumber\\
    &=&\alpha L_\xi(h'X)-(h'X\alpha)\xi-\alpha h'L_\xi X ~~~~~~\text{[since $h'\xi=0$]}\nonumber\\
    &=&\alpha(\nabla_\xi h')X-(h'X\alpha)\xi-\alpha h^2X+\alpha h'X. \label{e29}
    \end{eqnarray}
    Availing \eqref{e27}-\eqref{e29} in \eqref{e26}, we acquire
    \begin{eqnarray*}
        (L_VL_Vg)(X,Y)&=&(\xi\alpha)((X\alpha)\eta(Y)+(Y\alpha)\eta(X))+\alpha(X(\xi\alpha)\eta(Y)+Y(\xi\alpha)\eta(X))\nonumber\\
        &+&2(X\alpha)(Y\alpha)+2\alpha(\xi\alpha)g(h'X,Y)+2\alpha^2g((\nabla_\xi h')X,Y)\\
        &+&2\alpha^2g(h^2X,Y)-2\alpha^2g(h'X,Y)+\alpha (h'X\alpha)\eta(Y)+\alpha(h'Y\alpha)\eta(X).
    \end{eqnarray*}
    As $V$ is mixed Killing, using \eqref{e25} and the above, we have for all $X,Y\in\chi(M),$
    \begin{eqnarray}\label{e30}
        (\xi\alpha)((X\alpha)\eta(Y)+(Y\alpha)\eta(X))+\alpha(X(\xi\alpha)\eta(Y)+Y(\xi\alpha)\eta(X))+2(X\alpha)(Y\alpha)\nonumber\\
        +2\alpha(\xi\alpha)g(h'X,Y)+2\alpha^2g((\nabla_\xi h')X,Y)+2\alpha^2g(h^2X,Y)-2\alpha^2g(h'X,Y)\nonumber\\
        +\alpha (h'X\alpha)\eta(Y)+\alpha(h'Y\alpha)\eta(X)=f((X\alpha)\eta(Y)+(Y\alpha)\eta(X))+2\alpha fg(h'X,Y).
    \end{eqnarray}
    But then for $X=\xi$, we have
    \begin{eqnarray}
        3(\xi\alpha)(Y\alpha)+(\xi\alpha)^2\eta(Y)+\alpha\xi(\xi\alpha)\eta(Y)+\alpha Y(\xi\alpha)+\alpha(h'Y\alpha)=f(Y\alpha+(\xi\alpha)\eta(Y)).
    \end{eqnarray}
    Again for $Y=\xi$, we observe $\alpha(\xi\xi\alpha)=f(\xi\alpha)-2(\xi\alpha)^2$. Substituting this in the above, we obtain 
    \begin{equation*}
        \alpha(h'Y\alpha)=f(Y\alpha)-3(\xi\alpha)(Y\alpha)+(\xi\alpha)^2\eta(Y)-\alpha Y(\xi\alpha).
    \end{equation*}
    But then availing this in \eqref{e30} leads to
    \begin{eqnarray}
        -(\xi\alpha)[(X\alpha)\eta(Y)+(Y\alpha)\eta(X)]+(X\alpha)(Y\alpha)+\alpha(\xi\alpha-\alpha-f)g(h'X,Y)\nonumber\\+\alpha^2g(h^2X,Y)+\alpha^2g((\nabla_\xi h')X,Y)+(\xi\alpha)^2\eta(X)\eta(Y).
    \end{eqnarray}
    Replacing $X$ by $\phi X$ and observing that $(\nabla_\xi h')\phi X=-(\nabla_\xi h)X$, we have \ref{o1}.\\
    On the other hand, \ref{o2} follows when we take trace of the above equation.
\end{proof}

\begin{remark}\label{rem5}
 If $\alpha$ is constant along $\xi$, then from \ref{o2} of the above theorem it can be seen that $\alpha$ is constant on $M$ and $h=0$. Hence $V$ is mixed Killing if and only if $\xi$ is mixed Killing. On the other side, if $\alpha$ varies only along $\xi$, then for all $X$, we have $\phi X\alpha=0$. Hence $\alpha^2(\nabla_\xi h)X=\alpha^2\phi h^2X+\alpha(f+\alpha-\xi\alpha)hX$, follows from \eqref{o1}.
\end{remark}

\vspace{0.2cm}

Now suppose the Reeb vector field is mixed Killing and $V=\alpha\xi$ is pointwise collinear with $\xi$ for some smooth function $\alpha$. Then one might ask when $V$ will be mixed Killing. As $\xi$ is mixed Killing if and only if $h=0$, we have from \ref{o1} of the above theorem that $\operatorname{grad}\alpha=(\xi\alpha)\xi$. Thus the following corollary can be stated.

\vspace{0.2cm}
    
\begin{corollary}
    Consider an almost coK\"{a}hler manifold $M(\phi,\xi,\eta,g)$ with $\xi$ mixed Killing and a vector field $V$ pointwise collinear with $\xi$ on $M$. If $V$ is mixed Killing, then $\operatorname{grad}\alpha$ is pointwise collinear with $\xi$.
\end{corollary}

\vspace{0.2cm}

A vector field $V$ is said to admit an {\it infinitesimal contact transformation} if the global 1-form $\eta$ is invariant along $V$, in other words for some smooth function $\sigma$, $V$ satisfies
\begin{equation*}
    L_V\eta=\sigma\eta.    
\end{equation*}
Note that if an almost coK\"{a}hler manifold $M(\phi,\xi,\eta,g)$ admits an infinitesimal contact transformation $V$, then the above equation is equivalent to 
\begin{equation}\label{e34}
    \eta(\nabla_XV)=\sigma\eta(X)-g(h'V,X),
\end{equation}
for all vector fields $X$ on $M$. We now prove the following. 
\vspace{0.2cm}

\begin{theorem}
    If an almost coK\"{a}hler manifold $M(\phi,\xi,\eta,g)$ with $\xi$ mixed Killing admits an infinitesimal contact transformation $V$, then $\operatorname{grad}\sigma=(\xi\sigma)\xi$.
\end{theorem}
\begin{proof}
    According to Theorem \ref{t1}, as $\xi$ is mixed Killing, we have $h=0$ on $M$. Taking Lie derivative of $\eta(X)=g(X,\xi)$ along $V$ and since $V$ is an infinitesimal contact transformation, we have
    \begin{equation}\label{e35}
        (L_Vg)(X,\xi)=\sigma\eta(X)+g(X,\nabla_\xi V).
    \end{equation}
    But then again covariant derivative of the above along $\xi$ gives
    \begin{equation}\label{e36}
        (\nabla_\xi L_Vg)(X,\xi)=(\xi\sigma)\eta(X)+g(X,\nabla_\xi\nabla_\xi V).
    \end{equation}
    It can be easily seen from \cite{Y} that 
    $$(L_V\nabla)(X,Y)=\nabla_X\nabla_YV-\nabla_{\nabla_XY}V-R(X,V)Y.$$
    Evaluating this at $(\xi,\xi)$, and noting that $R(\xi,V)\xi=0$, one can see for any $Z\in\chi(M)$,
    \begin{equation}\label{e37}
        g((L_V\nabla)(\xi,\xi),Z)=g(\nabla_\xi\nabla_\xi V,Z).
    \end{equation}
    The following relation also holds on any Riemannian manifold \cite{Y},
    $$2g((L_V\nabla)(X,Y),Z)=(\nabla_XL_Vg)(Y,Z)+(\nabla_YL_Vg)(X,Z)-(\nabla_ZL_Vg)(X,Y).$$
    But then substituting $X=Y=\xi$ in the above, we get
    \begin{equation}\label{e38}
        2g((L_V\nabla)(\xi,\xi),Z)=2(\nabla_\xi L_Vg)(\xi,Z)-(\nabla_ZL_Vg)(\xi,\xi).
    \end{equation}
    We now compare \eqref{e37} and \eqref{e38} to obtain
    \begin{equation}\label{e39}
        2g(\nabla_\xi\nabla_\xi V,X)=2(\nabla_\xi L_Vg)(\xi,X)-(\nabla_XL_Vg)(\xi,\xi).
    \end{equation}
    Since $g(\xi,\xi)=1$, taking Lie derivative of this along $V$, we have
    \begin{eqnarray*}
        (L_Vg)(\xi,\xi)&=&-2g(L_V\xi,\xi)\\
                    &=&2g(\nabla_\xi V,\xi).~~~~~~~~~~~~~~~\text{[since $h'\xi=0$]}
    \end{eqnarray*}
    Covariant derivative of this along $X$ then yields 
    \begin{equation}\label{e40}
        (\nabla_XL_Vg)(\xi,\xi)=2\eta(\nabla_X\nabla_\xi V).
    \end{equation}
    On the other hand, evaluating \eqref{e34} for $X=\xi$, we have $\eta(\nabla_\xi V)=\sigma$. But then covariant derivative of this with respect to $X$ implies $\eta(\nabla_X\nabla_\xi V)=X\sigma$. Hence \eqref{e40} gives
    \begin{equation}\label{e41}
        (\nabla_XL_Vg)(\xi,\xi)=2X\sigma.
    \end{equation}
    Finally availing \eqref{e24},\eqref{e36} and \eqref{e41} in \eqref{e39}, we acquire $X\sigma=(\xi\sigma)\eta(X)$.
\end{proof}
\vspace{0.3cm}

\subsection{$\eta$-Einstein almost coK\"{a}hler manifold with $\xi$ mixed Killing}\label{ssec1}
Recall from Section \ref{sec3} that an $\eta$-Einstein almost coK\"{a}hler manifold $M^{2n+1}(\phi,\xi,\eta,g)$ satisfies \eqref{e16},
where $a$ and $b$ are smooth functions given by $a=\frac{1}{2n}\big(r+\tr(h^2)\big)$ and $b=-\frac{1}{2n}\big(r+(2n+1)\tr(h^2)\big)$. Also note that, if $\xi$ is mixed Killing, then $h=0$ and $r$ is constant along $\xi$. 
\vspace{0.2cm}

\begin{theorem}\label{t10}
    Consider an $\eta$-Einstein almost coK\"{a}hler manifold $M^{2n+1}(\phi,\xi,\eta,g)$ with $\xi$ mixed Killing. Then the manifold is of constant scalar curvature $r$ and $QX=\frac{r}{2n}\{X-\eta(X)\xi\}$. Moreover, it is Einstein if and only if it is Ricci flat.
\end{theorem}

\begin{proof}
We start with covariant derivatives of \eqref{e16} with $h=0$ and it gives
\begin{equation}\label{e42}
    (\nabla_XQ)Y=\frac{Xr}{2n}[Y-\eta(Y)\xi].
\end{equation}
Contraction of the above with respect to $X$, and using contracted Bianchi second identity, we have 
\begin{equation}\label{e43}
    (\operatorname{div}Q)Y=\frac{1}{2}Yr=\frac{1}{2n}[Yr+(\xi r)\eta(Y)].
\end{equation}
Because $\xi r=0$, we have $Yr=0$ for all $Y\in\chi(M)$. Hence $M$ has constant scalar curvature.
\end{proof}

\vspace{0.2cm}

\begin{remark}
    Note that the above theorem extends Proposition $3.2$ of \cite{W2}, where it is shown that if $\xi$ is Killing on any $\eta$-Einstein almost coK\"{a}hler manifold of dimension larger than 3, then $a$ and $b$ are constants. 
\end{remark}

\subsection{$(\kappa,\mu)$-almost coK\"{a}hler manifold with $\xi$ mixed Killing}\label{ssec2}
 Here we note the geometry of non-coK\"{a}hler $(\kappa,\mu)$-almost coK\"{a}hler manifold with $\xi$ as mixed Killing vector field. Recall from Section \ref{sec3} that $(\kappa,\mu)$-almost coK\"{a}hler manifolds are the almost coK\"{a}hler manifolds satisfying \eqref{e17}, where $\kappa,\mu$ are smooth functions. For $\kappa=0$, it is known that the manifold is coK\"{a}hler, and $\xi$ is Killing, hence mixed Killing. For non-coK\"{a}hler $(\kappa,\mu)$-almost coK\"{a}hler manifold, clearly $\xi$ can not be mixed Killing as $h\ne0$.

\vspace{0.2cm}
 
 \begin{theorem}
 On any non-coK\"{a}hler $(\kappa,\mu)$-almost coK\"{a}hler manifold $M(\phi,\xi,\eta,g)$, the Reeb vector field $\xi$ can not be mixed Killing.
 \end{theorem}
 
\vspace{0.2cm}

\begin{remark}
  The above discussions shows that on any $(\kappa,\mu)$-almost coK\"{a}hler manifold $M(\phi,\xi,\eta,g)$, $\xi$ is mixed Killing if and only if $\xi$ is Killing (and therefore the manifold is coK\"{a}hler).
\end{remark}
\vspace{0.2cm}

\begin{remark}
    Note that for $3$ dimensional $(\kappa,\mu)$-almost coK\"{a}hler manifold $M^3(\phi,\xi,\eta,g)$, the Ricci operator is given by \eqref{e20}. If $\xi$ is mixed Killing, then $h=0$ and hence the manifold $M^3$ is coK\"{a}hler.  Moreover $M^3$ is $\eta$-Einstein with Ricci operator given by (as $\kappa=0$)
    $$    QY=\frac{r}{2}\bigg(Y-\eta(Y)\xi\bigg).$$
    Then again $r$ is constant, and so $M^3$ is either $\mathbb{R}^3$ or a product space $\mathbb{R}\times K^2$, where $K^2$ is the K\"{a}hler surface of constant curvature, as discussed in Theorem \ref{t7}. 
\end{remark}

\vspace{0.3cm}

\subsection{Collective examples}\label{ssec3}

Now we present few examples of different types of almost coK\"{a}hler manifolds in this section.
\begin{enumerate}
    \item We start with a simple one. Consider $M=\mathbb{R}^3$ with the flat metric $g$. Clearly $\{\frac{\partial}{\partial x}, \frac{\partial}{\partial y}, \frac{\partial}{\partial z}\}$ forms an orthonormal basis for the tangent space $T_pM$ at any point $p\in M$. Consider $\xi=\frac{\partial}{\partial x}$, $\eta(.)=g(.,\frac{\partial}{\partial x})$ and define the $(1,1)$ tensor field $\phi$ as $\phi(\frac{\partial}{\partial x})=0$, $\phi(\frac{\partial}{\partial y})=\frac{\partial}{\partial z}$ and $\phi(\frac{\partial}{\partial z})=-\frac{\partial}{\partial y}$. Then it is easy to see that $(\phi,\xi, \eta,g)$ establishes a coK\"{a}hler structure on $M$. 

    If $V=x\frac{\partial}{\partial x}+(y-z)\frac{\partial}{\partial y}+(y+z)\frac{\partial}{\partial z}$, then $L_Vg=2g$ and hence $V$ is a homothetic vector field. Therefore $V$ is mixed Killing vector field on $M$ with mixed Killing factor $f=2$.

    \item Define a smooth manifold $M=\{(x,y,z)\in\mathbb{R}^3:z>0\}=\mathbb{R}^2\times \mathbb{R}_+$ and   the $(1,1)$ tensor field $\phi$ as $\phi\frac{\d}{\d x}=\frac{z^2}{e^{ax}}\frac{\d}{\d y},$ $\phi\frac{\d}{\d y}=-\frac{e^{ax}}{z^2}\frac{\d}{\d y}$ and $\phi\frac{\d}{\d z}=0$. Taking $\xi=\frac{\d}{\d z}$ and $\eta=dz$, it is shown that \cite{OD} $M(\phi,\xi,\eta,g)$  defines an almost coK\"{a}hler manifold where the compatible metric $g$ is given by $g=z^2dx^2+\frac{e^{2ax}}{z^2}dy^2+dz^2$. Consider an orthonormal global frame $\{e_1=\frac{1}{z}\frac{\d}{\d x},e_2=\frac{z}{e^{ax}}\frac{\d}{\d y}, e_3=\frac{\d}{\d z}\}$ on $M$ with the Lie brackets $[e_1,e_2]=-\frac{a}{z}e_2$, $[e_1,e_3]=\frac{1}{z}e_1$ and $[e_2,e_3]=-\frac{1}{z}e_2$. Then using the Koszul's formula, Levi-Civita connection of $g$ can be written as \cite{OD}
    \begin{eqnarray*}
        \nabla_{e_1}e_1=-\frac{1}{z}e_3,&~~~\nabla_{e_2}e_1=\frac{a}{z}e_2,&~~~\nabla_{e_3}e_1=0,\\
        \nabla_{e_1}e_2=0,&~~~\nabla_{e_2}e_2=-\frac{a}{z}e_1+\frac{1}{z}e_3,&~~~\nabla_{e_3}e_2=0,\\
        \nabla_{e_1}e_3=\frac{1}{z}e_1,&~~~\nabla_{e_2}e_3=-\frac{1}{z}e_2,&~~~\nabla_{e_3}e_3=0.
    \end{eqnarray*}
    It can be easily seen that $h\phi=\nabla\xi\ne0$ on $M$. In fact $he_1=\frac{1}{z}e_2$, $he_2=\frac{1}{z}e_1$ and $he_3=0$. Since $h\ne0$, we can affirm that $\xi$ is not mixed Killing on $M$. Moreover, any vector field $V=\alpha\xi$ with $\alpha$ constant along $\xi$ can not be mixed Killing on $M$ as $h\ne0$. On the other hand if $\alpha$ only varies along $\xi$, it is easy to see that (with the help of Remark \ref{rem5}) $V$ is mixed Killing only when $\alpha=0$. Therefore it can be asserted that any vector field pointwise collinear with $\xi$ can not be mixed Killing on $M$. As $M$ is not coK\"{a}hler, this discussions indirectly nod with Theorem \ref{t7}. It can be noted here that $\xi=e_3$ is not 2-Killing as $\nabla_{e_3}h\ne-2\phi h^2$, as depicted in Remark \ref{rem3}.

    \item A non-coK\"{a}hler almost coK\"{a}hler manifold $H(\phi,\xi,\eta,g)$ with flat K\"{a}hlerian leaves and with $\xi$ not mixed Killing can be obtained from {\it Example 3 of \cite{Ols1}}. Let $H=\{(x^0,x^k,x^{k'})\in\R^{2n+1}:k=1,2,...,n;k'=k+n\}$ be a semi-direct product of standard abelian Lie groups $\R$ and $\R^{2n}$ with the product rule
    $$x.y=(x^0,x^k,x^{k'}).(y^0,y^k,y^{k'})=(x^0+y^0,x^k+y^k\operatorname{exp}(-a_kx^0),x^{k'}+y^{k'}\operatorname{exp}(a_kx^0)),$$ for some real numbers $a_k:k=1,2,...,n$ such that $\sum a_k^2>0$.
    Consider the vector fields $e_0=\frac{\d}{\d x^0}$, $e_k=\operatorname{exp}(-a_kx^0)\frac{\d}{\d x^k}$, and $e_{k'}=\operatorname{exp}(a_kx^0)\frac{\d}{\d x^{k'}}$ which forms a basis for the Lie algebra of $H$ and the metric $g(e_i,e_j)=\delta_{ij}$. Defining $\xi=e_0$ with $\eta(e_0)=1$ and $\phi e_k=e_k'$, $\phi e_k'=-e_k$, $\phi e_0=0$, it can be seen that the manifold $H(\phi,\xi,\eta,g)$ constitutes an almost coK\"{a}hler structure. Non-zero Lie brackets can be defined on $H$ by $[e_0,e_k]=-a_ke_k$, $[e_0,e_{k'}]=a_ke_{k'}$ with non-zero components of the Levi-Civita connection
    \begin{eqnarray*}
        \nabla_{e_k}e_0=a_ke_k, ~~\nabla_{e_{k'}}e_0=-a_ke_{k'},~~ \nabla_{e_k}e_k=-a_ke_0 ~~\text{and}~~ \nabla_{e_{k'}}e_{k'}=a_ke_0.
    \end{eqnarray*}
    It can be easily seen that $he_k=a_ke_{k'}$ and $he_{k'}=a_ke_k$, $he_0=0$. Hence $e_0$ can not be mixed Killing. Theorem \ref{e8} shows that there doesn't exist any non-coK\"{a}hler almost coK\"{a}hler manifold  with $\xi$ mixed Killing. This very example supports to this fact. Although $\nabla_{e_0}h=0$ on $H$ but $e_0$ can not be 2-Killing as $\sum a_k^2>0$. 
    \item Let $M$ be a non-K\"{a}hler almost K\"{a}hler manifold. Then the product $\mathbb{R}\times M$ admits a non-coK\"{a}hler almost coK\"{a}hler manifold, for which the Reeb vector field $\xi$ is Killing \cite{Ols3}. Hence $\xi$ is mixed Killing.
    
    For instance, a non-coK\"{a}hler almost coK\"{a}hler manifold of dimension 5 with $h=0$ may be constructed by the product of the Kodaira-Thurston manifolds and a real line or a circle. Kodaira-Thurston manifolds are  compact 4 dimensional manifolds which are symplectic but not K\"{a}hler \cite{KU}. Clearly it also admits a almost K\"{a}hler structure due to its symplectic structure, in fact it is non-K\"{a}hler almost K\"{a}hler. Hence its product with a real line or a circle admits a non-coK\"{a}hler almost coK\"{a}hler manifold with $h=0$.
 \noindent
\end{enumerate}

\vspace{0.3cm}

\section{Conclusions} 
This research set out to study the curvature properties of a (semi-)Riemannian manifold along a particular direction, namely mixed Killing vector field. Bochner's theorem has been generalized in this setting. It is further studied in the framework of contact Riemmanian manifold and found out the necessary and sufficient condition for the characteristic vector field $\xi$ of almost coK\"{a}hler manifold to be mixed Killing. Towards this, we completely characterize three dimensional almost coK\"{a}hler manifolds with mixed Killing characteristic vector fields. We have also found the necessary condition for any vector field pointwise collinear with $\xi$ to be mixed Killing on almost coK\"{a}hler manifold. We also characterize the geometries of $\eta$-Einstein or $(\kappa,\mu)$-almost coK\"{a}hler manifolds with mixed Killing $\xi$. Several examples have been presented at the end. 

Remark \ref{rem2} has shown the nonlinearity of the flows of mixed Killing vector fields in one dimensional manifold. It is interesting to study further explicitly in higher dimensions (even in non-flat cases). This may explore a different sets of nonlinear partial differential equations which enrich the nonlinear characteristics of geometry.  The article mainly focuses on the geometry of almost coK\"{a}hler manifolds where $\xi$ is mixed Killing. Although it is of interest to classify the mixed Killing vector fields other than $\xi$ on this metric manifold. Different contact structures with mixed Killing vector fields may characterize the geometry towards its classifications. These unanswered questions may need further investigations in future.

Eventually it can be said that the geometries of the manifolds and their curvature estimations can be largely predicted from the mixed Killing vector fields of the manifolds.

\section{Acknowledgments} 

 Work of the author is financially supported by CSIR-UGC (Ref. No. 201610010610). The author is thankful to Prof. Arindam Bhattacharyya of Jadavpur University and Dr. Hemangi Madhusudan Shah of Harish-Chandra Research Institute for useful discussions and motivations toward this project.

\end{document}